\documentclass[article, 11pt]{amsart}

\usepackage{amssymb, amsthm, amsmath, color, tikz, natbib, hyperref}
\usetikzlibrary{arrows, shapes.geometric}
\usepackage[margin = 1in]{geometry}
\newtheorem{theorem}{Theorem}
\newtheorem{lemma}[theorem]{Lemma}
\newtheorem{corollary}[theorem]{Corollary}

\theoremstyle{remark}

\theoremstyle{definition}
\newtheorem{definition}[theorem]{Definition}
\newcommand{\oo}{\infty}
\newcommand{\R}{\mathbb{R}}

\newcommand{\PHI}{\varphi}
\newcommand{\DGM}[1]{(\text{Dgm}_{#1}, w_{#1})}
\newcommand{\HP}{\mathbb{R}_<^2}
\DeclareMathOperator{\ID}{id}
\DeclareMathOperator{\C}{cost}
\begin{document}

\title{Nonembeddability of Persistence Diagrams with $p>2$ Wasserstein Metric}

\author{Alexander Wagner}
\address{Department of Mathematics, University of Florida}
\email{wagnera@ufl.edu}
\urladdr{https://people.clas.ufl.edu/wagnera/}

\maketitle

\begin{abstract}
Persistence diagrams do not admit an inner product structure compatible with any Wasserstein metric. Hence, when applying kernel methods to persistence diagrams, the underlying feature map necessarily causes distortion. We prove persistence diagrams with the $p$-Wasserstein metric do not admit a coarse embedding into a Hilbert space when $p > 2$.
\end{abstract}

\section{Introduction}
\label{sec:intro}

The space of persistence diagrams with the $p$-Wasserstein metric is known not to admit an isometry into a Hilbert space for any $1 \leq p \leq \oo$ \citep{2019arXiv190301051T}. However, when applying kernel methods to persistence diagrams, the space of persistence diagrams is implicitly being mapped into a Hilbert space. Because of this, trying to lower-bound the non-trivial distortion caused by Hilbert space embeddings of persistence diagrams is of interest. \citet{MR3968607} investigated the question of bi-Lipschitz embeddings of persistence diagrams into separable Hilbert spaces. Bi-Lipschitz embeddings require that the metric in the domain is distorted at most linearly by the embedding. In contrast, coarse embeddings only require that the metric be distorted uniformly, but perhaps even super-exponentially. \citet{2019arXiv190505604B} proved that persistence diagrams with the $p$-Wasserstein metric do not admit a coarse embedding into a Hilbert space when $p = \oo$. The purpose of this paper is to extend this result to $p > 2$.

\section{Background}
\label{sec:background}

\subsection{The space of persistence diagrams}
\label{subsec:spaceofpd}
Persistence diagrams naturally arise as the output of persistent homology, which describes the changing homology of a one-parameter family of topological spaces. 

\begin{definition}
Denote $\{(x,y) \in \R^2\ |\ x < y\}$ by $\HP$. A \emph{persistence diagram} is a function from a countable set to $\HP$, i.e. $D: I \to \HP$.
\end{definition}

\begin{definition}
Suppose $D_1: I_1 \to \HP$ and $D_2: I_2 \to \HP$ are persistence diagrams. A \emph{partial matching} between them is a triple $(I_1', I_2', f)$ such that $I_1' \subseteq I_1$, $I_2' \subseteq I_2$, and $f: I_1' \to I_2'$ is a bijection. The \emph{$p$-cost of $f$} is denoted $\C_p(f)$ and defined as follows for $p < \oo$.
\[
\C_p(f) = \left( \sum_{i \in I_1'} \|D_1(i) - D_2(f(i))\|^p_{\oo} + \sum_{i \in I_1 \setminus I_1'} \left ( \frac{D_1(i)_y - D_1(i)_x}{2} \right)^p + \sum_{i \in I_2 \setminus I_2'} \left( \frac{D_2(i)_y - D_2(i)_x}{2} \right)^p\right)^{1/p}
 \]
That is, the $p$-cost of a partial matching is the $\ell^p$ norm of the sequence of distances between points paired by the partial matching and unpaired points with the diagonal in $\R^2$.
\end{definition}

\begin{definition}
Let $1\leq p < \oo$. If $D_1$, $D_2$ are persistence diagrams, define 
\[
\tilde{w}_p(D_1, D_2) = \inf \{ \C_p(f)\ |\ \text{f is a partial matching between $D_1$ and $D_2$} \}.
\]
Let $\DGM{p}$ denote the metric space of persistence diagrams $D$ that satisfy $\tilde{w}_p(D,\emptyset)< \oo$ modulo the relation $D_1 \sim D_2$ if $\tilde{w}_p(D_1, D_2) = 0$, where $\emptyset$ is shorthand for the unique persistence diagram with empty indexing set.
\end{definition}

\subsection{Coarse embeddings}
\label{subsec:coarseembedding}

Coarse embeddings are maps between metric spaces that allow distances to be distorted by amounts controlled by some fixed functions.

\begin{definition}
	A map $f:(X,d) \to (Y, d')$ is a \emph{coarse embedding} if there exists non-decreasing $\rho_-, \rho_+ :[0, \oo) \to [0, \oo)$ such that
	\begin{enumerate}
		\item $\rho_-(d(x,y)) \leq d'(f(x), f(y)) \leq \rho_+(d(x,y))$ for all $x,y \in X$, and
		\item $\lim \limits_{t \to \oo} \rho_-(t) = \oo$.
	\end{enumerate}
\end{definition}

\citet{MR2146202} proved that a metric space can be coarsely embedded into a Hilbert space iff every finite subset can be embedded in $\ell_2$ with the same distortion functions. \citet{MR2196037} subsequently gave a sufficient condition for a metric space to not admit a coarse embedding into a Hilbert space; this condition is satisfied in particular by $\ell_p$ for $p > 2$.

\begin{theorem}[{\citealt[Theorem 3.4]{MR2146202}}]
\label{thm:nowak}
A metric space $(X, d)$ admits a coarse embedding into a Hilbert space if and only if there exist non-decreasing functions $\rho_-, \rho_+: [0, \oo) \to [0, \oo)$ such that $\lim \limits_{t \to \oo} \rho_-(t) = \oo$ and for every finite subset $A \subseteq X$ there exists a map $f_A: A \to \ell_2$ satisfying
\[
\rho_-(d(x,y)) \leq \|f_A(x) - f_A(y) \|_2 \leq \rho_+(d(x,y))
\]
for every $x, y \in A$.
\end{theorem}

\begin{definition}
A basis $(e_n)_n$ for a Banach space $X$ is a normalized symmetric basis if 
\[
\| \sum \limits_{n} \theta_n a_n e_{\sigma(n)} \| = \| \sum \limits_{n} a_n e_n \|
\]for any choices of signs $\theta_n \in \{-1, +1\}$, permutation $\sigma : \mathbb{N} \to \mathbb{N}$	, and $\sum \limits_n a_n e_n \in X$.
\end{definition}

\begin{theorem}[{\citealt[Theorem 1]{MR2196037}}]
\label{thm:cecond}
Suppose that a Banach space $X$ has a normalized symmetric basis $(e_n)_n$ and that $\liminf \limits_{n \to \oo} n^{-\frac{1}{2}} \left \|\sum \limits_{i=1}^{n} e_i \right \| = 0$. Then $X$ does not coarsely embed into a Hilbert space.
\end{theorem}

\begin{corollary}[\citealt{MR2196037}]
\label{cor:lp>2}
	The space $\ell_p$ does not admit a coarse embedding into a Hilbert space for $p > 2$.
\end{corollary}
\section{Coarse embeddability of diagrams with $w_{p > 2}$ metric}
\label{sec:main}

In this section, we will show $\DGM{p}$ does not coarsely embed into a Hilbert space for $p > 2$. As in the case when $p = \oo$, the proof relies on the fact that $w_p$ is the infimum of the $\ell_p$ norm over all partial matchings.

\begin{lemma}
\label{lem:iso}
Let $d \in \mathbb{N}$. Every finite subset of $(\R^d, \| \cdot \|_p)$ isometrically embeds into $\DGM{p}$.
\end{lemma}

\begin{proof}
Let $A = \{ a^1, \dots, a^n \}$ be a finite subset of $\R^d$. Let 
\[
c > \max\{\| a^i \|_{\oo}, \| a^i - a^j \|_p\ |\ 1 \leq i,j \leq n\}.
\]
and consider the following map.
\begin{align*}
\PHI: A &\to \DGM{p} \\
a^i &\mapsto D_i : [d] \to \HP,\ k \mapsto \{(2c(k-1), 2c(k+1) + a^i_k)\}_{k=1}^{d}
\end{align*}

Formally, $\PHI(a^i)$ and $\PHI(a^j)$ are equivalence classes of persistence diagrams $D_i : [d] \to \HP$ and $D_j: [d] \to \HP$. Consider the partial matching $([d], [d], \ID_{[d]})$ between $D_i$ and $D_j$, i.e. $(2c(k-1), 2c(k+1) + a^i_k)$ is matched with $(2c(k-1), 2c(k+1) + a^j_k)$ for every $k \in [d]$. Observe that $\|D_i(k) - D_j(k)\|_{\oo} = |a^i_k - a^j_k|$ for every $k$, so the cost of this partial matching is $(\sum \limits_{k=1}^d |a^i_k - a^j_k|^p)^{\frac{1}{p}} = \| a^i - a^j \|_p$.

Suppose $I, J \subseteq [d]$ and $(I, J, f)$ is a different partial matching between $D_i$ and $D_j$. Then there exists a $k \in [d]$ such that either $k \notin I$ or $k \in I$ and $f(k) \neq k$. If $k \notin I$, then
\[
\C_p(f) \geq \frac{D_i(k)_y - D_i(k)_x}{2} = 2c + \frac{a^i_k}{2} > \frac{3c}{2}.
\]
If $k \in I$ and $f(k) = k' \neq k$, then
\[
\C_p(f) \geq \|(2c(k-1), 2ck + a^i_k) - (2c(k'-1), 2ck' + a^j_{k'})\|_{\oo} \geq 2c.
\]
Therefore, $\C_p(f) > c > \| a^i - a^j \|_p$. Hence, $([d], [d], \ID_{[d]})$ is the optimal partial matching and $w_p(\PHI(a^i), \PHI(a^j)) = \| a^i - a^j \|_p$, i.e. $\PHI$ is an isometric embedding.
	\end{proof}

\begin{theorem}
	$\DGM{p}$ does not coarsely embed into a Hilbert space for $2 < p < \oo$.
\end{theorem}

\begin{proof}
Fix $p > 2$ and suppose $\DGM{p}$ admits a coarse embedding into a Hilbert space. Then by Theorem \ref{thm:nowak}, there exist non-decreasing functions $\rho_-, \rho_+: [0, \oo) \to [0, \oo)$ such that $\lim \limits_{t \to \oo} \rho_-(t) = \oo$ and for every finite subset $S \subseteq \DGM{p}$ there exists a map $f_S: S \to \ell_2$ satisfying
\begin{equation}
\label{eq:ce}
\rho_-(w_p(x,y)) \leq \|f_S(x) - f_S(y) \|_2 \leq \rho_+(w_p(x,y))
\end{equation}
for every $x, y \in S$. Define $\tilde{\rho}_-: [0, \oo) \to [0, \oo)$ by $\tilde{\rho}_-(t) = \rho_-(\max(t-1, 0))$ and observe that $\tilde{\rho}_-$ is non-decreasing and $\lim \limits_{t \to \oo} \tilde{\rho}_-(t) = \oo$. We will show that for any finite subset $A \subseteq \ell_p$ there exists a map $g_A: A \to \ell_2$ satisfying 
\begin{equation}
\label{eq:g}
\tilde{\rho}_-(\|x - y\|_p) \leq \|g_A(x) - g_A(y) \|_2 \leq \rho_+(\|x - y\|_p)
\end{equation}
for every $x,y \in A$. This implies by Theorem \ref{thm:nowak} that $\ell_p$ coarsely embeds into a Hilbert space which contradicts Corollary \ref{cor:lp>2}.

For any $m \in \mathbb{N}$, define $\pi_m, \rho_m: \ell_p \to \ell_p$ by $\pi_m(x) = (x_1, \dots, x_m, 0, 0, \dots)$ and $\rho_m(x) = (0, \dots, 0, x_{m+1}, x_{m+2}, \dots)$. Let $A = \{a^1, \dots, a^n\}$ be a finite subset of $\ell_p$ and choose $m \in \mathbb{N}$ sufficiently large such that $\| \rho_m(a^i) \|_p \leq \frac{1}{2}$ for every $i = 1, \dots, n$. Then
\begin{align*}
\| a^i - a^j \|_p \geq \| \pi_m(a^i) - \pi_m(a^j) \|_p &\geq \|a^i - a^j\|_p - \| \rho_m(a^i) - \rho_m(a^j) \|_p \\
&\geq \|a^i - a^j \|_p - (\| \rho_m(a^i) \|_p + \| \rho_m(a^j) \|_p) \\
&\geq \|a^i - a^j \|_p - 1
\end{align*}

Since $\pi_m(A)$ is isometric to a finite subset of $(\R^m, \| \cdot \|_p)$, there exists an isometric embedding $\PHI: \pi_m(A) \to \DGM{p}$ by Lemma \ref{lem:iso}. Let $S = \PHI(\pi_m(A))$ and $f_S$ be as in \eqref{eq:ce}. Define $g_A: A \to \ell_2$ by $g_A = f_S \circ \PHI \circ \pi_m$. The following two series of inequalities show $g_A$ satisfies \eqref{eq:g}. In both cases, the first inequality follows from \eqref{eq:ce}, the equality follows from $\PHI$ being an isometric embedding, and the second inequality follows from the monotonicity of $\rho_+$ and $\rho_-$, respectively.
\begin{align*}
\|g_A(a^i) - g_A(a^j) \|_2 \leq \rho_+(w_p(\PHI \pi_m (a^i), \PHI \pi_m (a^j)))  = \rho_+(\|\pi_m (a^i) - \pi_m (a^j)\|_p) \leq \rho_+(\|a^i - a^j \|_p) \\
\|g_A(a^i) - g_A(a^j) \|_2 \geq \rho_-(w_p(\PHI \pi_m (a^i), \PHI \pi_m (a^j)))  = \rho_-(\|\pi_m (a^i) - \pi_m (a^j)\|_p) \geq \tilde{\rho}_-(\|a^i - a^j \|_p)
\end{align*}

\end{proof}

\end{document}